\documentclass[10pt]{article}
\usepackage{amsfonts,amsmath,amsthm,turnstile,color}
\usepackage{tikz}
\usepackage{pgf,tikz}
\usetikzlibrary{arrows}
\newcommand{\bR}{\mathbb{R}}
\newcommand{\bC}{\mathbb{C}}
\newcommand{\bN}{\mathbb{N}}
\newcommand{\bZ}{\mathbb{Z}}
\newcommand{\bK}{\mathbb{K}}
\newcommand{\scalar}[2]{\langle #1 , #2 \rangle}
\newtheorem{Tw}{Theorem}[section]
\newtheorem{Wn}[Tw]{Corollary}
\newtheorem{Uw}[Tw]{Remark}
\newtheorem{lem}[Tw]{Lemma}

\title{Loose edges}
\footnotetext{
      \begin{minipage}[t]{4.5in}{\small
       }
       \end{minipage}}
\author{Janusz Gwo\'zdziewicz and Beata Hejmej}

\begin{document}
\maketitle
\begin{abstract}
We consider formal power series in several variables with coefficients in arbitrary field
such that their Newton polyhedron has a loose edge. We show that if the symbolic restriction of 
the power series $f$ to such an edge is a product of two coprime polynomials, then $f$ 
factorizes in the ring of power series.  
\end{abstract}

\section{Introduction}

\textbf{Notation.} 
We denote by $\bR_{\geq0}$ (respectively $\bR_{>0}$) 
the set of nonnegative (respectively positive) real numbers. 
The symbol $\scalar{\cdot}{\cdot}$ denotes the standard scalar product.
We use a multi-index notation 
$\underline x^{\alpha}:=x_1^{\alpha_1}\cdots x_n^{\alpha_n}$ 
for $\alpha=(\alpha_1,\dots,\alpha_n)$. 

\medskip
We start from a quick reminder of convex geometry. 

Let $f\in \bK[\![x_1,\dots,x_n]\!]$,  $f=\sum a_{\alpha}\underline x^{\alpha}$ be a nonzero power series.
We define the {\em Newton polyhedron} $\Delta(f)$ as the convex hull of the set 
$\{ \alpha: a_{\alpha}\neq 0\} + \bR_{\geq 0}^n$.

The \textit{symbolic restriction} of $f$ to $A\subset \Delta(f)$ is defined 
as the power series
$$  f|_A=\sum_{\alpha\in A} a_{\alpha}\underline x^{\alpha}. $$

Given $\Delta=\Delta(f)$, for any $\xi\in\bR_{\geq 0}^n$ we call 
the set $$\Delta^{\xi} := \{\, a\in \Delta: \scalar{\xi}{a}=\min_{b\in \Delta} \scalar{\xi}{b}\,\}$$ 
a {\em face} of $\Delta$. 
A Newton polyhedron has a finite number of faces. 
A face $\Delta^{\xi}$ is compact if and only if $\xi\in \bR_{>0}^n$. 
A face of dimension 0  (respectively 1) is called  a {\em vertex} (respectively an {\em edge}).
Following \cite{lipkovski1988newton}, we call a  compact edge of a Newton polyhedron a {\it loose edge} 
if it is not contained in any compact face of dimension~$\geq 2$. 

Several Newton polyhedra are drawn in the pictures that follow. The segments marked in blue are loose edges

 
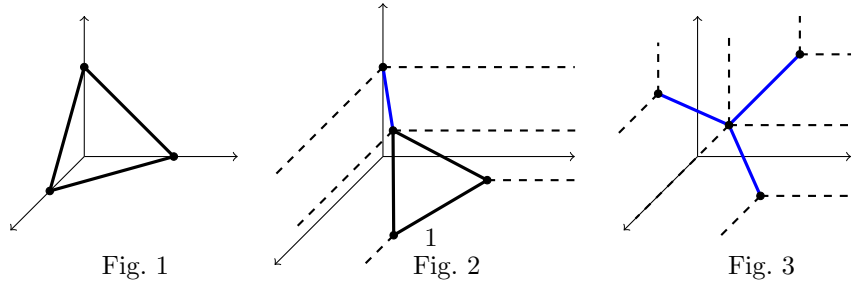
\begin{center}
\begin{tikzpicture}[scale=1.7]
\draw [->](0,0,0) -- (1.2,0,0); \draw[->](0,0,0) -- (0,1.1,0) ; 
\draw[->](0,0,0) -- (0,0,1.5);
\draw[very thick] (0,0.7,0) -- (0,0,0.7);
\draw[very thick] (0.7,0,0) -- (0,0,0.7);
\draw[very thick] (0,0.7,0) -- (0.7,0,0);
\node[draw,circle,inner sep=1pt,fill=black] at (0,0.7,0) {};
\node[draw,circle,inner sep=1pt,fill=black] at (0.7,0,0) {};
\node[draw,circle,inner sep=1pt,fill=black] at (0,0,0.7) {};
\node [below=1.2cm, align=flush center,text width=2cm] at (0.4,0,0)    {Fig.\ 1};
\end{tikzpicture}
\quad
\begin{tikzpicture}[scale=1.7]
\draw [->](0,0,0) -- (1.5,0,0); \draw[->](0,0,0) -- (0,1.2,0) ; 
\draw[->](0,0,0) -- (0,0,2.2);
\draw[thick, dashed] (0,0.7,0) -- (0,0.7,2.2);
\draw[thick, dashed] (0,0.7,0) -- (1.5,0.7,0);
\draw[very thick, color=blue] (0,0.7,0) -- (0.175,0.3,0.25);
\draw[very thick] (0.175,0.3,0.25) -- (0.7,0,1.6);
\draw[very thick] (0.175,0.3,0.25) -- (1,0,0.48);
\draw[very thick] (0.7,0,1.6) -- (1,0,0.48);
\draw[thick, dashed ] (0.7,0,1.6) -- (0.7,0,2.2);
\draw[thick, dashed] (1,0,0.48) -- (1.7,0,0.48);
\draw[thick, dashed] (0.175,0.3,0.25) -- (1.6,0.3,0.25);
\draw[thick, dashed] (0.175,0.3,0.25) -- (0.175,0.3,2.2);
\node[draw,circle,inner sep=1pt,fill=black] at (0,0.7,0) {};
\node[draw,circle,inner sep=1pt,fill=black] at (0.175,0.3,0.25) {};
\node[draw,circle,inner sep=1pt,fill=black] at (0.7,0,1.6) {};
\node[draw,circle,inner sep=1pt,fill=black] at (1,0,0.48) {};
\node [below=1.2cm, align=flush center,text width=2cm] at (0.5,0,0)    {Fig.\ 2};
\end{tikzpicture} 
\quad
\begin{tikzpicture}[scale=1.7]
\draw [->](0,0,0) -- (1.2,0,0); \draw[->](0,0,0) -- (0,1.1,0) ; 
\draw[->](0,0,0) -- (0,0,1.5);
\draw[very thick, color=blue] (0.8,0.8,0) -- (0.4,0.4,0.4);
\draw[very thick, color=blue] (0.8,0,0.8) -- (0.4,0.4,0.4);
\draw[very thick, color=blue] (0,0.8,0.8) -- (0.4,0.4,0.4);
\draw[thick, dashed] (0.4,0.4,0.4) -- (0.4,1.1,0.4);
\draw[thick, dashed] (0.4,0.4,0.4) -- (0.4,0.4,2.3);
\draw[thick, dashed] (0.4,0.4,0.4) -- (1.35,0.4,0.4);
\draw[thick, dashed] (0.8,0.8,0) -- (1.2,0.8,0);
\draw[thick, dashed] (0.8,0.8,0) -- (0.8,1.1,0);
\draw[thick, dashed] (0,0.8,0.8) -- (0,0.8,1.7);
\draw[thick, dashed] (0,0.8,0.8) -- (0,1.2,0.8);
\draw[thick, dashed] (0.8,0,0.8) -- (0.8,0,1.7);
\draw[thick, dashed] (0.8,0,0.8) -- (1.5,0,0.8);
\node[draw,circle,inner sep=1pt,fill=black] at (0.4,0.4,0.4) {};
\node[draw,circle,inner sep=1pt,fill=black] at (0.8,0.8,0) {};
\node[draw,circle,inner sep=1pt,fill=black] at (0.8,0,0.8) {};
\node[draw,circle,inner sep=1pt,fill=black] at (0,0.8,0.8) {};
\node [below=1.2cm, align=flush center,text width=2cm] at (0.5,0,0)    {Fig.\ 3};
\end{tikzpicture}
\end{center}

The Newton polyhedron in Figure 1 does not have any loose edge. 
This is the typical situation. 
\newpage
The Newton polyhedron in Figure 2 has a loose edge with the end point 
at $(0,0,d)$.   The Weierstrass polynomials  $f\in \bK[\![x_1,\dots,x_n]\!][z]$ 
such that $\Delta(f)$ is of this type were studied in \cite{parusinski2012abhyankar} 
and in \cite{rond2017irreducibility} where $\Delta(f)$ is called an 
\textit{orthant associated polyhedron}

In Figure 3 all compact edges are loose. A Newton polyhedron with this property 
is called in \cite{perez2000singularites} a {\em polygonal Newton polyhedron}. 
Notice that the term polygonal Newton polyhedron can be 
a bit misleading since the union of compact edges in Figure 3 is not homeomomorphic 
to any polygon. 

Every compact edge of a plane Newton polyhedron is loose as illustrated in Figure 4. 

\begin{center}
\begin{tikzpicture}[
     scale = 0.6,
     foreground/.style = {  thick },
     background/.style = { dashed },
     line join=round, line cap=round
   ]
   \draw[fill=black, opacity=0.1] (4,0)--(4.9,0)--(4.9,4.9)--(0,4.9)--(0,4)--(1,2)--(2,1)--cycle;
   \draw[foreground,->] (0,0)--+(5,0);
   \draw[foreground,->] (0,0)--+(0,5);
   \draw[very thick, color=blue] (1,2)--(2,1);
   \draw[very thick, color=blue] (0,4)--(1,2);
   \draw[very thick, color=blue] (2,1)--(4,0);
   \draw (1,-0.15) node[below] {$ $};
   \draw (-0.15,1) node[left] {$ $};
   \foreach \x in{} {
    \foreach \y in{}{
     \draw[fill, opacity=0.9]  (\x,\y) circle (0.5pt);
     }
    }; 
  \node[draw,circle,inner sep=1pt,fill=black] at (0,4) {};
\node[draw,circle,inner sep=1pt,fill=black] at (1,2) {};
\node[draw,circle,inner sep=1pt,fill=black] at (2,1) {};
\node[draw,circle,inner sep=1pt,fill=black] at (4,0) {};  
\node [below=0.4cm, align=flush center,text width=4cm] at (2.2,0)    {Fig.\ 4};
\end{tikzpicture}
\end{center}

Below are the main results of the paper. 
\begin{Tw} \label{t1}
Let $f\in \bK[\![x_1,\dots, x_n]\!]$ be a formal power series with coefficients in a field $\bK$.
Assume that the Newton polyhedron $\Delta(f)$ has a  loose edge $E$.
If~$f|_E$ is a product of two relatively prime polynomials $G$ and $H$, 
where $G$ is not divided by any variable, then there exist powers series  $g$, $h$ such that 
$f=gh$ and  $g|_{E_1}=G$, $h|_{E_2}=H$ for some $E_1$, $E_2$ such that $E=E_1+E_2$. 
\end{Tw}

In the above theorem $E_1$ is a loose edge of $\Delta(g)$ parallel to $E$ and $E_2$ is a compact face of $\Delta(h)$ 
which is a loose edge parallel to $E$ or a vertex 
if $H$ is a monomial. 

\begin{Wn}\label{w1}
Assume that the Newton polyhedron of $f\in\bK[\![x_1,\dots, x_n]\!]$ 
has a loose edge and at least three vertices. 
Then $f$ is not irreducible. 
\end{Wn}

\begin{Wn}\label{w2}
Assume that the Newton polyhedron of $f\in\bK[\![x_1,\dots, x_n]\!]$ 
has a loose edge $E$.
If $f$ is irreducible,
then $E$ is the only compact edge of $\Delta(f)$ and $f|_E=cF^k$,
 where $F$ is an irreducible polynomial. 
Moreover, if $\bK$ is algebraically closed, then 
$f|_E=(a\underbar x^{\alpha}+b\underbar x^{\beta})^k$ 
with a primitive lattice vector $\alpha-\beta$.
\end{Wn}

We will say that a segment $E\subset\bR^n$ is \textit{descendant} if $E$ is parallel to some vector 
$ c=(c_1,\dots,c_n)$ such that $c_i\geq 0$ for $i=1,\dots,n-1$ and $c_n<0$.

\begin{Tw} \label{t43}
Let $f\in\bK[\![x_1,\dots,x_{n-1}]\!][x_n]$. 
Assume that the Newton polyhedron $\Delta(f)$ has a loose and  descendant edge $E$. 
If $f|_E$ is a product of two relatively prime polynomials  $G$ and $H$, 
where  $G$ is monic with respect to $x_n$, then there exist  uniquely determined 
$g,h\in \mathbb{K}[\![x_1,\dots,x_{n-1}]\!][x_n]$ such that $f=gh$, 
$g$ is a monic polynomial with respect to $x_n$
and  $g|_{E_1}=G$, $h|_{E_2}=H$ for some  $E_1$, $E_2$ such that $E=E_1+E_2$. 
\end{Tw}

\section{Proofs}
At the beginning of this section we establish some results about loose edges. 

\begin{lem}\label{Lem:edge}
Let $\Delta$ be a Newton polyhedron with a loose edge $E$ that has ends 
$a$, $b\in\mathbb{R}^n$.
Then for every $c\in \Delta$ and every $\xi\in\mathbb{R}_{\geq0}^n$ such that 
$\scalar{\xi}{a} = \scalar{\xi}{b}$ one has $\scalar{\xi}{c} \geq \scalar{\xi}{a}$.
\end{lem}

\begin{proof}
Let $V$ be the  set of vertices of $\Delta$.  
If $V=\{a,b\}$ then Lemma~\ref{Lem:edge} follows easily. 
Hence in the rest of the proof assume that there exists a vertex of $\Delta$ different from $a$ or $b$ 
and consider the function 
$$\psi(\xi) =  \min_{v \in V\setminus\{a,b\}} \scalar{\xi}{v} - \scalar{\xi}{a},\quad
\xi\in\mathbb{R}_{\geq0}^n.$$ 
Since the set of vertices is finite, this function is well defined and continuous. 

Since $E$ is a compact face of $\Delta$, there exists $\xi_0\in \mathbb{R}_{>0}^n$ such that 
$\scalar{\xi_0}{a} = \scalar{\xi_0}{b}$ and 
$\scalar{\xi_0}{c} > \scalar{\xi_0}{a}$  for all $c\in \Delta\setminus E$. 
This yields $$\psi(\xi_0)>0. $$

Suppose that there exist $c\in\Delta$ and $\xi_1\in\mathbb{R}_{\geq0}^n$ such that 
$\scalar{\xi_1}{c} < \scalar{\xi_1}{a} = \scalar{\xi_1}{b}$.  Since 
$c=\lambda_1v_1+\cdots+\lambda_s v_s+ z$, 
 for some $z\in\mathbb{R}_{\geq0}^n$, $v_1,\dots,v_s\in V$ and $\lambda_i\geq 0$ such that $\lambda_1+\cdots+\lambda_s=1$,
we get $\scalar{\xi_1}{v} < \scalar{\xi_1}{a}$ for 
at least one vertex $v\in V$.  Thus $$\psi(\xi_1)<0.$$

It follows from the above that there exist $\xi$ in the segment, joining $\xi_0$ and~$\xi_1$, 
such that $\psi(\xi)=0$. 
It means that there exists $v\in V\setminus\{a,b\}$ such that 
$\scalar{\xi}{v}=\scalar{\xi}{a}=\scalar{\xi}{b}\leq \scalar{\xi}{d}$ for all $d\in\Delta$. 
 This implies that $\Delta^{\xi}$ is a compact (since $\xi\in\mathbb{R}_{>0}^n$) 
 face of dimension $\geq 2$, which contains $E$.
\end{proof}

\begin{lem}
Let $\Delta$ be a Newton polyhedron with a loose edge $E$ that has end points 
$a=(a_1,\dots, a_n)$, $b=(b_1,\dots, b_n)$.
If $\min(a_1,b_1)=\dots=\min(a_n,b_n)=0$, 
then $a$ and $b$ are the only vertices of $\Delta$. \label{lc}
\end{lem}

\begin{proof}
By assumption there exist two nonempty and disjoint subsets $A$, $B$ of the set of indices $\{1,\dots, n\}$ such that 
$a_i>0$ for $i\in A$, $a_i=0$ for $i\in \{1,\dots, n\}\setminus A$,
$b_j>0$ for $j\in B$ and $b_j=0$ for $j\in \{1,\dots, n\}\setminus B$.

Let $c=(c_1,\dots,c_n)$  be an arbitrary point of $\Delta$ different from $a$ and $b$. 
For any $i\in A$, $j\in B$ consider the vector $\xi_{ij}$, which has 
only two nonzero entries: $1/a_i$ at place $i$ and  $1/b_j$ at place $j$
($i\neq j$ since $A\cap B=\emptyset$). Then 
$\scalar{\xi_{ij}}{a}=\scalar{\xi_{ij}}{b}=1$ and 
$\scalar{\xi_{ij}}{c}=c_i/a_i+c_j/b_j$. 
By Lemma~\ref{Lem:edge} we get $c_i/a_i+c_j/b_j \geq 1$.
It follows that 
$$ \min_{i\in A} c_i/a_i + \min_{j\in B} c_j/b_j \geq 1 .$$

Choose constants  $\lambda\geq0$, $\mu\geq0$ such that 
$\lambda\leq \min_{i\in A} c_i/a_i$, $\mu\leq \min_{j\in B} c_j/b_j$ 
and $\lambda+\mu=1$.
Then $c=\lambda a+\mu b+z$ for some $z\in \mathbb{R}_{\geq0}^n$, hence 
$c$ cannot be a vertex of $\Delta$. 
\end{proof}

\begin{lem}\label{Lem:base}
Let $c\in\bZ^n$ be a point with at least one positive and at least one negative coordinate. 
Then there exists such a basis $\xi_1$,\dots,$\xi_{n}$ of the vector space~$\bR^n$ that 
$\xi_i\in\bN^n$ for $i=1,\dots,n$ and $\scalar{\xi_i}{c}=0$ for $i=1,\dots,n-1$.
\end{lem}

\begin{proof}
For any basis $\xi_1$, \dots, $\xi_{n}$ of~$\bR^n$ set 
$w=(w_1,\dots,w_n):=(\scalar{\xi_1}{c},\dots,\scalar{\xi_n}{c})$.
By the hypothesis of the lemma, for the standard basis 
$\xi_1=(1,0,\dots,0)$, \dots, $\xi_n=(0,\dots,0,1)$ of $\bR^n$ there exist 
$w_i, w_j\neq 0$ of opposite signs.
We will gradually modify the standard basis  until we reach a basis such that only one 
coordinate of $w$ is nonzero. 
Below we outline the algorithm. 

\begin{itemize}
\item[1.] If there are only two nonzero entries $w_j$, $w_k$ of $w$ and $w_j+w_k=0$, then 
replace $\xi_k$ by $\xi_k+\xi_j$ and stop.
\item[2.] Let $j$ be the index such that $|w_j|=\min\{\,|w_i|: w_i\neq 0,\,i\in\{1.\dots,n\}\,\}$.
\begin{itemize}
\item[(a)] If there exists $w_k$ of sign opposite to this of $w_j$, such that  $|w_k|>|w_j|$, 
        then replace $\xi_k$ by $\xi_k+\xi_j$,
\item[(b)] otherwise if there are at least two entries $w_k$, $w_l$ of sign opposite 
         to this of $w_j$, then replace $\xi_k$ by $\xi_k+\xi_j$, 
\item[(c)] otherwise set $j:=k$ and perform step (a). 
\end{itemize}
\item[3.] Go to step 1.
\end{itemize}

After every loop of the above algorithm $w_k$ is replaced by $w_k+w_j$. Hence 
the number $\sum_{i=1}^n |w_i|$ decreases and the algorithm must terminate.
\end{proof}

We encourage the reader to apply the algorithm from the proof in a simple case, 
for example for $c=(2,3,-4)$. 

\medskip 
From now on up to the end of this section we fix a loose edge $E$. 
Let $c$ be a primitive lattice vector parallel to $E$.
Applying Lemma~\ref{Lem:base} to $c$ we find ${n-1}$~linearly independent vectors
$\xi_1$,\dots, $\xi_{n-1}\in \bN^n$ which are orthogonal to $E$.
For every monomial $\underline{x}^{\alpha}$ we set 
$\omega(\underline{x}^{\alpha}) := 
(\scalar{\xi_1}{\alpha},\,\dots,\scalar{\xi_{n-1}}{\alpha})$. 
We call this vector a {\it weight} of $\underline{x}^{\alpha}$. 
Since 
$\omega(\underline{x}^{\alpha}\underline{x}^{\beta}) =
\omega(\underline{x}^{\alpha})+\omega(\underline{x}^{\beta})$ 
for any monomials $\underline{x}^{\alpha}$, $\underline{x}^{\beta}$, 
the ring $\mathbb{K}[x_1,\dots,x_n]$ turns into a graded ring 
$\bigoplus_{w\in\mathbb{N}^{n-1}} R_{w}$,
where $R_{w}$ is spanned by monomials of weight $w$.

All monomials of $R_w$ are of the form $\underline{x}^{a+ic}$ where 
$\underline{x}^{a}$ is a fixed monomial of $R_w$ and $i$ is an integer. This shows that $R_w$ is a finite dimensional vector space over $\mathbb{K}$
since there is only a finite number of integers $i$ such that all coordinates of $a+ic$ are nonnegative. One can happen that $\dim R_w=0$. 
Take for example the weight $\omega(x_1^{\alpha_1}x_2^{\alpha_2})=3\alpha_1+2\alpha_2$.
Then there is no  monomial of weight 1, hence $R_1\subset \mathbb{K}[x_1,x_2]$ is a zero-dimensional vector space. 

Let $M\subset \mathbb{N}^{n-1}$ be the set of weights satisfying the condition: 
$z\in M$ if and only if there exists $\alpha\in\mathbb{Z}^n$ 
such that $\omega(\underline{x}^{\alpha})=z$ and $\scalar{\xi}{\alpha}\geq 0$ for all 
$\xi\in\mathbb{R}_{\geq0}^n$ which are orthogonal to $E$. 
Observe that $M$ is closed under addition. Moreover for any  $w\in\mathbb{N}^{n-1}$ such that $\dim R_w>0$ we have $w\in M$.

\begin{lem}
Let $w\in\mathbb{N}^{n-1}$ and $z\in M$. 
Assume that~$R_w$ contains two coprime monomials. 
Then $$\dim R_{w+z}=\dim R_w+ \dim R_z-1.$$\label{l1}
\end{lem}

\begin{proof}
 For any $v\in\mathbb{N}^{n-1}$ 
 the dimension of the vector space $R_v$ is equal to the number of 
 monomials ${\underline x}^{\alpha}$ such that $\omega({\underline x}^{\alpha})=v$, hence 
 is equal to the number of lattice points in the set 
$$l_v = \{\,\alpha\in \bR_{\geq 0}^n : 
(\scalar{\xi_1}{\alpha},\dots, \scalar{\xi_{n-1}}{\alpha})=v\,\} . $$
Notice that $l_v$ is the intersection of the straight line 
$\{a+tc: t\in \mathbb{R}\}$, where $\omega({\underline x}^{a})=v$ and 
$c$ is a primitive lattice vector parallel to the edge $E$, 
with $\bR_{\geq 0}^n$.

By the assumption, $R_w$ contains two coprime monomials 
${\underline x}^{a}$ and ${\underline x}^{b}$. 
Hence $\min(a_1,b_1)=\dots=\min(a_n,b_n)=0$ which yields the partition 
of $\{1,\dots,n\}$ to three sets 
$A=\{i\in\{1,\dots,n\}: a_i=0, b_i>0\}$, 
$B=\{i\in\{1,\dots,n\}: a_i>0, b_i=0\}$ and
$C=\{i\in\{1,\dots,n\}: a_i=0, b_i=0\}$. 
Since $A$ and $B$ are nonempty,  $a$ and $b$ 
are the endpoints of the segment $l_w$. 
We may assume without loss of generality that the vector $c$ is pointed 
in the direction of $b-a$. Then if  $\dim R_w=r+1$, then 
the lattice points  of $l_w$ are $a$, $a+c$, \dots, $b=a+rc$.
 
If the lattice points of $l_z$  are 
$d$, $d+c$, \dots, $d+s c$ then the lattice points of 
$l_{w+z}$,  are $a+d$, $a+d+c$, \dots, $a+d+(r+s)c$.
(see Figure 5).  This ends the proof in the case $\dim R_z>0$.

Now suppose that $\dim R_z=0$. Let ${\underline x}^{d}$ be any monomial  
with integer exponents such that  $\omega({\underline x}^{d})=z$.
Replacing this monomial by ${\underline x}^{d+kc}$, for suitably chosen integer $k$, 
we may assume that  $d_{i_0}<0$ for some  $i_0\in A$ and 
$d_i+c_i\geq 0$  for all $i\in A$. By the assumption that $z\in M$ we get inequalities
$d_{i_0}/c_{i_0}-d_j/c_j\geq 0$ for $j\in B$. Hence $d_j\geq 0$ for all $j\in B$. 
By the same argument we have  $d_j\geq0$ for $j\in C$. Notice that 
$d_{j}+c_j<0$ for at least one $j\in B$, otherwise all exponents of ${\underline x}^{d+c}$ 
would be nonnegative. All this information  implies that  $a+d+c$, $a+d+2c$, \dots, $a+d+rc$ are the 
only lattice points of $l_{w+z}$ which finishes the proof. 

\end{proof}

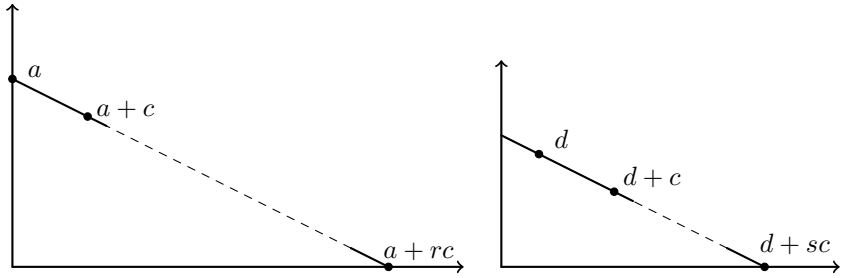
\begin{center}
\begin{tikzpicture}[
     scale = 1,
     foreground/.style = {  thick },
     background/.style = { dashed },
     line join=round, line cap=round
   ]
   \draw[foreground,->] (0,0)--+(6,0);
   \draw[foreground,->] (0,0)--+(0,3.5);
 \draw[thick] (5,0)--(4.5,0.25);
   \draw[dashed] (4.5,0.25)--(1.25,1.875);
   \draw[thick] (1.25,1.875)--(0,2.5);
   \draw (1,-0.15) node[below] {$ $};
   \draw (-0.15,1) node[left] {$ $};
   \foreach \x in{} {
    \foreach \y in{}{
     \draw[fill, opacity=0.9]  (\x,\y) circle (0.5pt);
     }
    }
\node[draw,circle,inner sep=1pt,fill=black] at (0,2.5) {};
\node[draw,circle,inner sep=1pt,fill=black] at (1,2) {};
  \node[draw,circle,inner sep=1pt,fill=black] at (5,0) {};
\node at (0.3,2.6) {$a$}; 
\node at (1.5,2.1) {$a+c$}; 
   \node at (5.4,0.2) {$a+rc$};   
  
  [
     scale = 1,
     foreground/.style = {  thick },
     background/.style = { dashed },
     line join=round, line cap=round
   ]
  \draw[foreground,->] (6.5,0)--+(4.5,0);
   \draw[foreground,->] (6.5,0)--+(0,2.75);
   \draw[thick] (10,0)--(9.5,0.25);
   \draw[dashed] (9.5,0.25)--(6.5,1.75);
   \draw[thick] (6.5,1.75)--(8.25,0.875);
   \draw (1,-0.15) node[below] {$ $};
   \draw (-0.15,1) node[left] {$ $};
   \foreach \x in{} {
    \foreach \y in{}{
     \draw[fill, opacity=0.9]  (\x,\y) circle (0.5pt);
     }
    }
\node[draw,circle,inner sep=1pt,fill=black] at (7,1.5) {};
\node[draw,circle,inner sep=1pt,fill=black] at (8,1) {};
\node[draw,circle,inner sep=1pt,fill=black] at (10,0) {};
\node at (7.3,1.7) {$d$};
 \node at (8.5,1.2) {$d+c$};
   \node at (10.4,0.3) {$d+sc$}; 
   
    \end{tikzpicture}
  \end{center}
  
  \begin{center}
  \begin{tikzpicture}
    [
       scale = 1,
     foreground/.style = {  thick },
     background/.style = { dashed },
     line join=round, line cap=round
   ]
   \draw[foreground,->] (0,0)--+(10,0);
   \draw[foreground,->] (0,0)--+(0,5);
   \draw[thick] (8.5,0)--(8,0.25);
   \draw[dashed] (8,0.25)--(5,1.75);
    \draw[dashed] (6.75,0.875)--(1.75,3.375);
    \draw[thick] (0,4.25)--(1.75,3.375);
   \draw[thick] (6.75,0.875)--(5,1.75);
   \draw (1,-0.15) node[below] {$ $};
   \draw (-0.15,1) node[left] {$ $};
   \foreach \x in{} {
    \foreach \y in{}{
     \draw[fill, opacity=0.9]  (\x,\y) circle (0.5pt);
     }
    }
    \node[draw,circle,inner sep=1pt,fill=black] at (1.5,3.5) {};
 \node[draw,circle,inner sep=1pt,fill=black] at (0.5,4) {};
\node[draw,circle,inner sep=1pt,fill=black] at (5.5,1.5) {};
  \node[draw,circle,inner sep=1pt,fill=black] at (6.5,1) {};
    \node[draw,circle,inner sep=1pt,fill=black] at (8.5,0) {};
   \node at (1.1,4.2) {$a+d$}; 
    \node at (2.3,3.7) {$a+d+c$};
   \node at (9.5,0.3) {$a+d+(r+s)c$};
   \node [below=0.6cm, align=flush center,text width=6cm] at (4.5,0)    {Fig.\ 5};
  \end{tikzpicture}
 \end{center}

\begin{lem} \label{l3}
Let $G\in R_w$ and $H\in R_z$ be coprime polynomials. 
If $G$ is not divisible by any monomial  then for every $i\in M$ 
$$ G R_{z+i} + H R_{w+i} = R_{w+z+i}.
$$
\end{lem}

\begin{proof}
Consider the following exact sequence 
$$ 0\longrightarrow{} R_i\stackrel{\Phi}{\longrightarrow} 
R_{z+i}\times R_{w+i} \stackrel{\Psi}{\longrightarrow} R_{w+z+i},
$$ 
where 
$\Phi\colon \eta \mapsto (\eta H,-\eta G)$ and $\Psi\colon (\psi,\varphi)\mapsto \psi G+\varphi H$. 
The assumption on $G$ implies that $R_w$ satisfies the hypothesis of Lemma~\ref{l1}.
Hence we get 
$\dim {\rm Im}\,\Psi = \dim R_{z+i}+\dim R_{w+i} - \dim R_i = 
\dim R_{z+i}+ (\dim R_w+\dim R_i - 1) - \dim R_i = 
\dim R_w+\dim R_{z+i} -1 = \dim R_{w+z+i}$, 
which implies that $\Psi$ is surjective.
\end{proof}

\begin{proof}[Proof of Theorem \ref{t1}]
Since all monomials of $f|_E$ have the same weight,
it follows from the the equality $f|_E=GH$ that $G\in R_{w}$ and $H\in R_{z}$ 
for some $w$, $z\in {\mathbb N}^{n-1}$.
Let ${\underline x}^{\alpha}$ be any monomial which appears 
with nonzero coefficient in the power series~$f$ 
and ${\underline x}^{\alpha_0}$ be a fixed monomial of $f|_E$. 
By Lemma~\ref{Lem:edge} we have $\scalar{\xi}{\alpha} \geq \scalar{\xi}{\alpha_0}$
for every $\xi\in\mathbb{R}_{\geq0}^n$ which is orthogonal to $E$. 
This yields $\omega({\underline x}^{\alpha-\alpha_0})\in M$. Since 
$\omega({\underline x}^{\alpha_0})=w+z$, we get 
$f=\sum_{i\in M} f_{w+z+i}$, where $f_{w+z+i}\in R_{w+z+i}$.

Let $g_w:=G$ and $h_{z}:=H$.  Then $f_{w+z}=g_{w}h_{z}$. 
Let us set  $M$ in degree-lexicographic order.
Using Lemma~\ref{l3} we can find recursively 
 $g_{w+i}\in R_{w+i}$ and $h_{z+i}\in R_{z+i}$ 
such that 
$$g_{w}h_{z+i}+h_{z}g_{w+i}=f_{w+z+i}-F_i,
$$ 
where 
$$F_i=\sum_{\stackrel{k+l=i,}{k,l<i}}g_{w+k}h_{z+l}.$$ 
If $g:=\sum_{i\in M} g_{w+i}$ 
and $h:=\sum_{i\in M} h_{z+i}$, then $f=gh$.

Let $\xi=\xi_1+\dots+\xi_{n-1}$. Then for  $E_1:=\Delta(g)^{\xi}$ and 
$E_2:=\Delta(h)^{\xi}$ we have $g|_{E_1}=g_{w}$, $h|_{E_2}=h_{z}$ 
and $E=E_1+E_2$. 
\end{proof}

\begin{Uw}
The assumption of Theorem \ref{t1} that  $G$ is not divisible by any variable  cannot be omitted. 
Consider the power series $f=(x_3^2+x_1x_2)(x_3+x_1 x_2)$. 
Its Newton polyhedron has a loose edge $E$ with endpoints $(1,1,1)$, $(2,2,0)$ and 
$f|_E=x_1x_2(x_3+x_1x_2)$. The irreducible factors of $f$ are $f_1=x_3^2+x_1x_2$ and 
$f_2=x_3+x_1 x_2$. Hence $f$ cannot be a product of power series $g$, $h$ such that 
$g|_{E_1}=x_2(x_3+x_1x_2)$ and $h|_{E_2}=x_1$.
\end{Uw}

\begin{proof}[Proof of Corollary \ref{w1}]
Assume that $a=(a_1,\dots, a_n)$, $b=(b_1,\dots,b_n)$ 
are the ends of a loose edge $E$ of the Newton polyhedron $\Delta(f)$. 
Since $\Delta(f)$ has at least three vertices,  Lemma \ref{lc} implies that 
$c=(\min(a_1,b_1),\dots,\min(a_n,b_n))$ has at least one nonzero coordinate.
The monomials $\underline{x}^{a}$ and $\underline{x}^{b}$ 
appear in the polynomial $f|_E$ with nonzero coefficients and their greatest common divisor equals
$\underline{x}^{c}$.
Thus $\underline{x}^{-c}\cdot f|_E$ is not divisible  by any variable, 
so $\underline{x}^{-c}\cdot f|_E$ and $\underline{x}^c$ are relatively prime.
Using Theorem \ref{t1} we obtain that $f$ is not irreducible.
\end{proof}

%


\begin{lem} \label{l4}
Let $G\in R_w$ and $H_j\in R_{z_j}$ for $j=1,2$. 
Assume that for every $i\in M$ 
$$ G R_{z_j+i} + H_j R_{w+i} = R_{w+z_j+i}
$$
for $j=1,2$. Then for every $i\in M$ 
$$ G R_{z_1+z_2+i} + H_1H_2 R_{w+i} = R_{w+z_1+z_2+i}.
$$
\end{lem}

\begin{proof} By assumptions of the lemma we get
\begin{eqnarray*}
&\phantom{=}& G R_{z_1+z_2+i} + H_1H_2 R_{w+i} = 
G (R_{z_1+z_2+i}+H_1 R_{z_2+i}) + H_1H_2 R_{w+i}  \\
&=&  G R_{z_1+z_2+i} + H_1 ( G R_{z_2+i} + H_2 R_{w+i} ) \\ 
&=& G R_{z_1+z_2+i} + H_1R_{w+z_2+i} = R_{w+z_1+z_2+i} \,.
\end{eqnarray*}
\end{proof}

Assume that $E$ is a descendant loose edge. By definition, there exists  a lattice vector  
$c=(c_1,\dots,c_n)$ parallel to $E$ such that $c_i\geq 0$ for $i=1,\dots,n-1$ and $c_n<0$.
Let $R_w':=R_w\cap \bK[x_1,\dots,x_{n-1}]$ for $w\in\bN^{n-1}$. Since every line parallel to
$E$ intersects $\bR^{n-1}\times \{0\}$ transversely, the dimension of  $R_w'$ is 0 or 1. 

We claim that for every $w\in M$, $z\in \mathbb{N}^{n-1}$, $0\neq H\in R_z'$ and $\psi\in R_{w+z}'$  
one has $\psi/H\in R_w'$.
To prove this claim it is enough to consider $H=\underline{x}^{\alpha}\in  R_z'$ 
and $\psi=\underline{x}^{\beta}\in R_{w+z}'$. Then $\omega(\underline{x}^{\beta-\alpha})=w$.
Denote $v_i:=e_i-\frac{c_i}{c_n}e_n\in\mathbb{R}_{\geq 0}^n$, where $e_1,\dots,e_n$ is the standard basis of $\mathbb{R}^n$. Since every vector $v_i$ is orthogonal to  $E$ and $w\in M$, we have $0\leq\langle v_i,\beta-\alpha\rangle = \beta_i-\alpha_i$ for $i=1,\dots,n-1$. 
Thus  $\psi/H=\underline{x}^{\beta-\alpha}$ is a monomial with nonnegative exponents. 

\begin{lem} \label{l5}
Let $G\in R_w$ and $H\in R_z$ be coprime polynomials. 
If $G$ is monic with respect to $x_n$,
 then for every $i\in M$ 
\begin{equation} \label{eq2}
G R_{z+i} + H R_{w+i} = R_{w+z+i}.
\end{equation}
\end{lem}

\begin{proof} First, we prove~(\ref{eq2}) in the special case
$G=x_n$ and $H\in R_z'$. 

Every $F\in R_{w+z+i}$ can be written in the form 
$F=x_n\phi + \psi$ where $\phi, \psi$ are polynomials and 
$\psi$ does not depend on $x_n$. Then 
$\phi\in R_{z+i}$ and $\psi=H\psi'$ where $\psi' = \psi/H\in R_{w+i}'$.
This gives~(\ref{eq2}) in the special case. 
All remaining cases follow from Lemma~\ref{l3} and Lemma~\ref{l4}. 
\end{proof}

\begin{proof}[Proof of Theorem \ref{t43}]
Proceeding as in the proof of Theorem \ref{t1}, but using Lemma~\ref{l5} instead of Lemma~\ref{l3}
we obtain $\overline{g},\overline{h}\in\bK[\![x_1,\dots,x_{n-1},x_n]\!]$ such that 
$f=\overline{g}\overline{h}$, $\overline{g}|_{E_1}=G$ and $\overline{h}|_{E_2}=H$ 
for some segments $E_1,E_2$, where $E=E_1+E_2$. 
The assumptions that the loose edge $E$ is descendant and the polynomial $G$  
is monic with respect to  $x_{n}$ imply that the Newton polyhedron of $\overline{g}$ 
has a~vertex $(0,\dots,0,d)$ for some positive integer $d$.
Therefore $\overline{g}(0,\dots,0,x_n)=v x_n^d$ for some $v\in\bK[\![x_n]\!]$ such that $v(0)\ne 0$. 
It means that $\overline{g}$ fulfills assumptions of the Weierstrass Preparation Theorem, which implies that $\overline{g}=u \hat{g}$, 
where $u$ is a power series such that $u(0)\ne 0$
and $\hat{g}\in \bK[\![x_1,\dots,x_{n-1}]\!][x_n]$  is a Weierstrass polynomial. 
Putting $g=\hat{g}$ and $h=u\overline{h}$ we get $f=gh$.  
Since we can also obtain $h$ using the polynomial division of $f$ by $g$
in the ring $\bK[\![x_1,\dots,x_{n-1}]\!][x_n]$,
we conclude that $h$ is a polynomial with respect  to~$x_n$. 
\end{proof}

\section{Relation with known results}
\medskip
Corollary~\ref{w2}  generalizes to $n>2$ variables the following well-known fact.

\begin{Tw}
Assume that a power series $f\in \bC[\![x,y]\!]$ written as a sum of homogeneous polynomials 
$f=f_d+f_{d+1}+\dots$, where the subindex means the degree,  is irreducible. 
Then $f_d$ is a power of a linear form.
\end{Tw}

\medskip
Below we quote some notations of \cite{artal2015high} and Lemma~A1 of that paper.

\medskip
Let $\bK$ be a field and fix a weight $\omega(x^iy^j) := ni + mj$ for $n,m\in\bN$. 
Given $0 \neq F \in\bK[\![x, y]\!]$, we will consider its decomposition in 
$\omega$-quasihomogeneous forms
$$ F(x, y) = F_{a+b}(x, y) + F_{a+b+1}(x, y) + \dots\,, $$
where the subindex means the $\omega$-weight.

\begin{Tw}
Asume that $F_{a+b}(x, y) = f_a(x, y)g_b(x, y)$, where
$f_a$, $g_b \in\bK[x, y]$ are quasihomogeneous and coprime. 
Then, there exist
$$f, g \in\bK[\![X, Y ]\!],\quad f = f_a + f_{a+1} + \dots\,,\quad  g = g_b + g_{b+1} + \dots
$$
such that $F = fg$. Moreover if $f_a$ is an irreducible polynomial, 
then $f$ is an irreducible power series.
\end{Tw}
Theorem~\ref{t1} can be seen as a  generalization of this results. 

\medskip
Corollary~\ref{w2} generalizes the following result of \cite{barroso2005decomposition}

\begin{Tw}
If $\phi\in \bC\{x_1,\dots ,x_n\}$ is irreducible and has a polygonal Newton polyhedron $\Delta(\phi)$,
then the polyhedron $\Delta(\phi)$ has only one compact edge $E$ and the polynomial $\phi|_E$ is a power of an irreducible polynomial. 
\end{Tw}

The term polygonal Newton polyhedron in the statement of the above theorem is used for Newton polyhedra such that all their compact edges are loose. 

\medskip
Theorem \ref{t43} generalizes the main result of \cite{rond2017irreducibility} quoted below.

\begin{Tw}
 Let $P(Z)\in k[\![x_1,\dots,x_n]\!][Z]$ be a monic Weierstrass polynomial. Assume
that P(Z) has an orthant associated polyhedron and that $P|_{\Gamma}\in k[x_1,\dots,x_n,Z]$ 
is the product of two coprime monic polynomials 
$S_1$, $S_2 \in k[x_1,\dots,x_n,Z]$, respectively, of degree $d_1$ and $d_2$. 
Then there exist two monic polynomials $P_1$, $P_2 \in k[\![x_1,\dots,x_n]\!][Z]$, respectively,
of degrees $d_1$ and $d_2$ in $Z$ such that \\
i) $P = P_1P_2$, \\
ii) there is at least one $i\in \{1, 2\}$ such that $P_i$ has an orthant associated polyhedron
and if $\Gamma_i$ denotes the compact face of $\Delta(P_i)$ containing the points
$(0, \dots,0,d_i)$, then $P_i|_{\Gamma_i} = S_i$ and $\Gamma_i$ is parallel to $\Gamma$.
\end{Tw}

The term  {\it orthant asociated polyhedron} in the statement of the above theorem means 
a Newton polyhedron that has a loose edge with endpoint $(0,\dots,0,d)$. 

\vspace{1cm}
\quad
\begin{minipage}[t]{2.5in}
{\small   Janusz Gwo\'zdziewicz\\
Institute of Mathematics\\
Pedagogical University of Cracow\\
Podchor\c a{\accent95 z}ych 2\\
PL-30-084 Cracow, Poland\\
e-mail: gwozdziewicz@up.krakow.pl}
\end{minipage}
\begin{minipage}[t]{2in}
{\small   Beata Hejmej\\
Institute of Mathematics\\
Pedagogical University of Cracow\\
Podchor\c a{\accent95 z}ych 2\\
PL-30-084 Cracow, Poland\\
e-mail: bhejmej1f@gmail.com}
\end{minipage}


\begin{thebibliography}{22}

\bibitem{artal2015high} E. Artal Bartolo, I. Luengo, and A. Melle-Hern\'andez, 
{\em High-school algebra of the theory of dicritical divisors:  
Atypical fibres for special pencils and polynomials,} 
Journal of Algebra and Its Applications, 14.09 (2015), 1540009.

\bibitem{barroso2005decomposition} E. Garc\'{\i}a Barroso and P. Gonz\'alez P\'erez, 
{\em Decomposition in bunches of the critical locus of a quasi-ordinary map,} 
Compos. Math. 141 no.\  2 (2005), 461--486. DOI: 10.1112/S0010437X04001216.

\bibitem{perez2000singularites} P. Gonz\'alez P\'erez,
{\em Singularit\'es quasi-ordinaires toriques et poly\`edre de Newton du discriminant,}
Canad. J. Math. 52 (2000), 348--368.

\bibitem{lipkovski1988newton} A. Lipkovski, 
{\em  Newton polyhedra and irreducibility,} Math. Z. 199 (1988), no.~1, 119--127.

\bibitem{parusinski2012abhyankar} A. Parusi\'nski and G. Rond, 
{\em The Abhyankar-Jung Theorem}, Journal of Algebra 365 (2012)  29--41.

\bibitem{rond2017irreducibility}  G. Rond and B. Schober, 
{\em An irreducibility criterion for power series},
Proceedings of the American Mathematical Society,
Volume 145, Number 11, November 2017, Pages 4731--4739, doi.org/10.1090/proc/13635.

\end{thebibliography}
\end{document}